\documentclass{amsart}
\usepackage{amssymb,amsrefs,mathrsfs}

\newcommand\Z{\mathbb{Z}}
\newcommand\N{\mathbb{N}}
\newcommand\R{\mathbb{R}}
\newcommand\Q{\mathbb{Q}}

\newtheorem{lemma}{Lemma}[section]

\newtheorem{proposition}[lemma]{Proposition}

\newtheorem{corollary}[lemma]{Corollary}
\newtheorem{example}[lemma]{Example}

\newtheorem{maintheorem}{Theorem}

\theoremstyle{definition}
\newtheorem{remark}[lemma]{Remark}
\newtheorem{definition}[lemma]{Definition}

\newcommand\grig{G_{012}}
\newcommand\wrwr{\operatorname{\wr\wr}}

\begin{document}
\title{Imbeddings into groups of intermediate growth}

\author{Laurent Bartholdi}
\address{L.B.: Mathematisches Institut, Georg-August Universit\"at, G\"ottingen, Germany}
\email{laurent.bartholdi@gmail.com}

\author{Anna Erschler}
\address{A.E.: C.N.R.S., D\'epartement de Math\'ematiques, Universit\'e Paris Sud, Orsay, France}
\email{anna.erschler@math.u-psud.fr}

\date{June 23, 2014}

\dedicatory{To Pierre de la Harpe, who introduced us to the beauty and diversity of the world of infinite groups, in gratitude}

\thanks{The work is supported by the ERC starting grant 257110
  ``RaWG'', the ANR ``DiscGroup: facettes des groupes discrets'', the
  Centre International de Math\'ematiques et Informatique, Toulouse,
  and the Institut Henri Poincar\'e, Paris}

\begin{abstract}
  Every countable group that does not contain a finitely generated
  subgroup of exponential growth imbeds in a finitely generated group
  of subexponential word growth.
\end{abstract}
\maketitle

\section{Introduction}

A classical result by Higman, Neumann and
Neumann~\cite{higman-n-n:embed} states that every countable group
imbeds in a finitely generated group. It was then shown that many
properties of the group can be inherited by the imbedding: in
particular, solvability (Neumann-Neumann~\cite{neumann-n:embed}),
torsion (Phillips~\cite{phillips:embedding}), residual finiteness
(Wilson~\cite{wilson:embeddingrf}), and amenability
(Olshansky-Osin~\cite{olshanskii-osin:qiembedding}).

Seen the other way round, these results show that there is little
restriction, apart from being countable, on the subgroups of a
finitely generated group.  \smallskip

A finitely generated group $G$ has \emph{polynomial growth} if there
is a polynomial function $p(n)$ bounding from above the number of
group elements that are products of at most $n$ generators, has
\emph{subexponential growth} if $p(n)$ may be chosen subexponential in
$n$, and has \emph{intermediate growth} if $G$ has subexponential but
not polynomial growth.

By a theorem of Gromov~\cite{gromov:nilpotent}, groups of polynomial
growth are virtually nilpotent, so all its subgroups are finitely
generated (see e.g.~\cite{mann:howgroupsgrow}*{Corollary~9.10}). On
the other hand, there are groups of intermediate growth such as the
``first Grigorchuk group''~\cite{grigorchuk:growth} with infinitely
generated subgroups. We are therefore led to ask which groups may
appear as subgroups of a group of subexponential growth.

\subsection{Main result}
Let us say that a group has \emph{locally subexponential growth} if
all of its finitely generated subgroups have subexponential growth.
Clearly, if $G$ has subexponential growth then all its subgroups have
locally subexponential growth. Our main result shows that this is the
only restriction:
\begin{maintheorem}\label{thm:imbed}
  Let $B$ be a countable group of locally subexponential growth. Then
  there exists a finitely generated group of subexponential growth in
  which $B$ imbeds as a subgroup.
\end{maintheorem}
Furthermore, this group may be assumed to have two generators, see
Remark~\ref{rem:2gen}, and to contain $B$ in its derived subgroup.

In contrast, there exist nilpotent (and even abelian) countable groups
that do not imbed into finitely generated nilpotent groups. Gromov's
theorem mentioned above has the consequence that there exist countable
groups of locally polynomial growth that do not imbed in groups of
polynomial growth. Mann noted
in~\cite{mann:howgroupsgrow}*{Corollary~9.11} that torsion-free groups
locally of polynomial growth of bounded degree are also virtually
nilpotent.

It is a tantalizing open question to understand which properties are
shared by groups of intermediate growth and by nilpotent and virtually
nilpotent groups.  It is clear that a group of intermediate growth
cannot contain non-abelian free subgroups or even free subsemigroups.
Groups of intermediate growth were constructed by Grigorchuk in
\cite{grigorchuk:growth}, and his first example, known as the ``first
Grigorchuk group'', admits a pair of dilating endomorphisms with
commuting images.  This property can be viewed as a higher dimensional
analogue of groups with dilation; and any group admitting a dilation
has polynomial growth.

We may also ask which groups may appear as subgroups of a specific
group of intermediate growth such as the first Grigorchuk group
$\grig$. For example, $\grig$ is known to contain every finite
$2$-group, and all its subgroups are countable, residually-$2$ and
have locally smaller growth.

There are other restrictions, apart from these obvious ones, for a
countable group to be imbedded as a subgroup of a generalised
Grigorchuk group. For example, only a finite number of primes appears
as exponents in a Grigorchuk group~\cite{grigorchuk:pgps}; see
also~\cite{bartholdi-g-s:bg}*{\S3.6}.  Extensions of Grigorchuk groups
constructed by the authors in~\cite{bartholdi-erschler:permutational}
admit a larger class of possible subgroups, but some restrictions
appear nevertheless. In particular, Theorem~\ref{thm:imbed} gives the
first groups of subexponential growth containing $\Q$.

\subsection*{Acknowledgments}
The authors are grateful to the anonymous referees for their comments,
which helped improve the presentation of the article.

\section{Sketch of the proofs}

The original imbedding result by
Higman-Neumann-Neumann~\cite{higman-n-n:embed} mentioned in the
introduction proceeds by a sequence of ``HNN extensions''.  We recall
the later construction by Neumann-Neumann~\cite{neumann-n:embed},
which uses wreath products rather than HNN extensions. The
\emph{unrestricted} wreath products of two groups $H,G$ is the group
$H\wrwr G=H^G\rtimes G$, the split extension of the set of maps $G\to
H$ by $G$, where the action of $g\in G$ on $f\colon G\to H$ is
${}^gf(x)=f(xg)$. The Neumann-Neumann construction proceeds in two
steps:

(i) starting with a countable group $B$, one imbeds it into a countable
subgroup $G$ of the unrestricted wreath product $B\wrwr\Z$ in such a
way that $B$ is imbedded into the commutator group $[G,G]$. The group
$G$ is generated by $\Z$ and, for all $b\in B$, the function
$f_b\colon\Z\to B$ defined by $f_b(m)=b^m$. Denoting by $t$ the
generator of $\Z$, we see that $[t,f_b]$ is the constant function $b$;
so $B$ is in fact imbedded in $[t,G]$.

(ii) starting with a countable group $G$, one imbeds the commutator
subgroup $[G,G]$ into a two-generated subgroup $W$ of the unrestricted
wreath product of $G\wrwr\Z$. More generally, one constructs
imbeddings into $G\wrwr P$ for a finitely generated group
$P$. Denoting a generating set of $G$ by $\{b_1,b_2,\dots\}$, the
group $W$ is generated by $P$ and $f\colon P\to G$ with $f(x_i)=b_i$
along a sparse-enough sequence $(x_i)_{i\ge1}$ of elements of $P$. In
fact, since it suffices in~(i) to imbed $[t,G]$ in $W$, one sets
$f(1)=t$ and the exact requirement on the sequence $(x_i)$ is:
$x_i\neq1$ for all $i$; all $x_i$ are distinct; and $x_i
x_j\not\in\{1,x_k\}$ for all $i,j,k\in\N$. One then sees that
$[f,f^{x_i^{-1}}]$ is a function supported only at $1$, with value
$[t,b_i]$ there. This is the imbedding of $[t,G]$.

The combination of both steps imbeds $B$ into the finitely generated
group $W$. If $B$ is solvable, then so are $B\wrwr\Z$ and $G$; and
similarly, if $G$ is solvable, then so are $G\wrwr\Z$ and $W$.

This construction may be applied to an arbitrary countable group $B$,
but some properties of $B$, such as amenability, may be lost along the
way. Olshansky and Osin introduce in~\cite{olshanskii-osin:qiembedding}
the following slightly stronger condition on $(x_i)_{i\ge1}$: by
definition, a \emph{parallelogram} in a sequence $(x_i)_{i\ge1}$ is a
quadruple of elements $p_1\neq p_2\neq p_3\neq p_4\neq p_1$, each
belonging to $\{x_i\}$, such that $p_1p_2^{-1}p_3p_4^{-1}=1$. A
sequence is \emph{parallelogram-free} if it contains no
parallelogram. They show that, if $(x_i)$ is parallelogram-free, then
the group $W$ is obtained from $G$ and $P$ by extensions, subgroups,
quotients and directed limits, so in particular is amenable as soon as
$G$ and $P$ are amenable.

They also modify slightly step~(i), by defining rather $f_b(m)=b$ for
$m\ge0$ and $f_b(m)=1$ for $m<0$; then $[t,f_b]$ is the function
supported at $0$ with value $b$ there, and the group $G=\langle
t,f_b\colon b\in B\rangle$ is also obtained from $B$ and $\Z$ by
elementary operations, so is amenable as soon as $B$ is amenable.

Note that the group $W$ contains the standard wreath product $B\wr\Z$,
so always has exponential growth.

\subsection{Imbedding in groups of subexponential growth}
Our goal is, starting from a countable group $H$ of locally
subexponential growth, to construct a finitely generated group $C$ of
subexponential growth. We exhibit analogues of steps~(i) and~(ii)
among \emph{permutational wreath products}. Given groups $H,G$ and an
action of $G$ on a set $X$, the \emph{unrestricted permutational
  wreath product} is $H\wrwr_X G=H^X\rtimes G$, and the
\emph{restricted permutational wreath product} is the extension of
finitely supported functions $X\to H$ by $G$.

Our previous work~\cite{bartholdi-erschler:permutational} gives a
criterion, in terms of \emph{inverted orbits}, that guarantees that
the restricted permutational wreath product $W=H\wr_X G$ has
subexponential growth as soon as $H$ and $G$ have subexponential
growth. The inverted orbit of a point $x\in X$ under a word
$w=g_1\dots g_n$ in $G$ is the set $\{x g_1\dots g_n,x g_2\dots
g_n,\dots,x g_n, x\}$. If its cardinality may be bounded sublinearly
in $n$, then $W$ has subexponential growth. We compare subgroups
$\langle G,f\rangle$ of the unrestricted wreath product with $W$ to
bound its growth.

Ad step~(i), we show in Proposition~\ref{prop:imbed'} that for every
group $B$ there exists a group $G$ that is a directed union of finite
extensions of finite powers of $B$ and such that $[G,G]$ contains
$B$. In particular, if $B$ has locally subexponential growth, so does
$G$.

Ad step (ii), we need to consider separately the group $P$ and the set
$X$ on which it acts. As a replacement for parallelogram-free
sequences, we introduce \emph{rectifiable} sequences, which are
sequences $(x_i)$ in $X$ such that, for all $i\neq j$, there exists
$g\in P$ with $x_i g=x_j$ and $x_k g\neq x_\ell$ for all $\ell\neq
k\neq i$. We show that such sequences exist in the action of the first
Grigorchuk group on the orbit of a ray, and more generally for all
``weakly branched'' groups.

The next step in the proof is an argument controlling the growth of a
subgroup of the form $W=\langle G,f\rangle\le B\wrwr_X G$, for a
function $f\colon X\to B$ with sparse-enough (but infinite!)
support. The rectifiability of the sequence $(x_i)$ guarantees that
functions with singleton support and arbitrary values in $[B,B]$
belong to $W$. Using the sparseness of the support of $f$, we show
that balls in $W$ can be approximated by balls in subgroups of
restricted wreath products $\langle S\rangle\wr_X f$ for finite
subsets $S$ of $B$. By~\cite{bartholdi-erschler:permutational}, these
restricted wreath products have subexponential growth. We recall that,
in general, a limit in the Cayley topology of groups of subexponential
growth may have exponential
growth~\cite{bartholdi-erschler:orderongroups}*{Theorem~C}; the Cayley
topology on the space of finitely generated groups is the topology in
which groups are close if their labeled Cayley graphs agree on a large
ball. We control more precisely the approximation of $W$ so that the
growth estimates pass to the limit. Finally, in contrast with standard
wreath products, the space $X$ is not homogeneous, so and extra
condition of stabilisation of balls around the $x_i$ is required (even
to ensure that $W$ be amenable).

\section{Imbedding in the derived subgroup}\label{ss:imbedG'}

Let $B$ be a group. We call a group $G$ \emph{hyper-$B$} if it is a
directed union of finite extensions of finite powers of $B$. In this
section, our goal is to prove the following proposition.

\newcommand\wrf{\operatorname{\wr^{\text{\upshape f.v.}}}}

\begin{proposition}\label{prop:imbed'}
  Let $B$ be a group. Then there exists a hyper-$B$ group $G$ such
  that $[G,G]$ contains $B$ as a subgroup. In particular, if $B$ has
  locally subexponential growth, then so does $G$.

  If $B$ is infinite, then $G$ may furthermore be supposed to have the
  same cardinality as $B$.
\end{proposition}

In order to prove Proposition~\ref{prop:imbed'}, we first introduce
the following notation. For groups $H,U$ we denote by
\[H\wrf U=\{(\phi,u)\in H^U\times U\colon\#\phi(U)<\infty\}
\]
the subgroup of the unrestricted wreath product $H^U\rtimes U$
in which the function $U\to H$ takes finitely many values. Note that it is
a subgroup, because if $(\phi,u)^{-1}(\phi',u')=(\phi'',u^{-1}u')$
then $\phi''(U)\subseteq\phi(U)^{-1}\phi'(U)$ is finite.

\begin{lemma}\label{lem:hyperhyper}
  Let $G$ be a hyper-$B$ group, and let $H$ be a hyper-$G$ group. Then
  $H$ is hyper-$B$.
\end{lemma}
\begin{proof}
  Consider $h\in H$; then $h$ belongs to a finite extension of a
  finite power of $G$, which may be assumed of the form $G\wr F$ for a
  finite group $F$. Let us write $h=\phi f$ with $\phi\colon F\to G$
  and $f\in F$; then $\phi(f)$ belongs for all $f\in F$ to a finite
  extension of a finite power of $B$, which can be assumed to be the
  same for all $f$. This extension may be assumed to be of the form
  $B\wr E$ for a finite group $E$. It follows that $h$ belongs to
  $B\wr_{E\times F}(E\wr F)$, a finite extension of a finite power of
  $B$; so $H$ is hyper-$B$.
\end{proof}

\begin{lemma}\label{lem:hyperB}
  If $H$ is a hyper-$B$ group and $U$ is locally finite, then $H\wrf
  U$ is a hyper-$B$ group.
\end{lemma}
\begin{proof}
  We first show that $H\wrf U$ is hyper-$H$. By hypothesis, $U$ is a
  directed union of finite subgroups $E$. The partitions $\mathscr
  P_0$ of $U$ into finitely many parts also form a directed poset; and
  for every such partition $\mathscr P_0$ and every finite subgroup $E\le U$
  there exists a finite partition $\mathscr P$ of $U$ that is
  invariant under $E$ and refines $\mathscr P_0$, namely the wedge (=
  least upper bound) of all $E$-images of $\mathscr P_0$.

  Consider now the directed poset of pairs $(E,\mathscr P)$ consisting
  of finite subgroups $E\le U$ and $E$-invariant partitions of $U$.
  Consider the corresponding subgroups $H^{\mathscr P}\rtimes E$ of
  $H\wrf U$. If $(E,\mathscr P)\le(E',\mathscr P')$ then $H^{\mathscr
    P}\rtimes E$ is naturally contained in $H^{\mathscr P'}\rtimes
  E'$, so these subgroups of $H\wrf U$ form a directed poset, which
  exhausts $H\wrf U$.

  It follows that $H\wrf U$ is a hyper-$H$ group, and we are done by
  Lemma~\ref{lem:hyperhyper}.
\end{proof}

\begin{lemma}\label{lem:CB}
  Let $B$ be a group. Then there exists a subgroup $C$ of $B$,
  containing $[B,B]$, such that $B/C$ is torsion and $C/[B,B]$ is free
  abelian.
\end{lemma}
\begin{proof}
  $B/[B,B]\otimes_\Z\Q$ is a $\Q$-vector space, hence has a basis,
  call it $X$. It generates a free abelian group $\Z X$ within
  $B/[B,B]$, whose full preimage in $B$ we call $C$. Then
  $B/C\otimes_\Z\Q=0$ so $B/C$ is torsion.
\end{proof}

We set up the following notation for the proof of
Proposition~\ref{prop:imbed'}. We choose a subgroup $C\le B$ as in
Lemma~\ref{lem:CB} and write $T:=B/C$. We choose a basis $X$ of
$C/[B,B]$, for every $x\in X$ we choose an element $b_x\in C$
representing it, and we define a homomorphism $\theta_x\colon
C\to\langle b_x\rangle\subseteq B$, trivial on $[B,B]$, by
$\theta_x(b_x)=b_x^{-1}$ and $\theta_x(b_y)=1$ for all $y\neq x\in
X$. In particular, we have for all $b\in C$
\[b\cdot\prod_{x\in X}\theta_x(b)\in[B,B]
\]
and the product is finite.

We write $\pi\colon B\to T$ the natural projection, and define
inductively a set-theoretic section $\sigma\colon T\to B$ as
follows. Since $T$ is torsion, it is locally finite, hence may be
written as a directed union $T=\bigcup_{\alpha\in I} T_\alpha$ of
finite groups, over a well-ordered directed set $I$. Assume $\sigma$
has already been defined on
$T'_\alpha:=\bigcup_{\beta<\alpha}T_\beta$. Choose a transversal
$T''_\alpha$ of $T'_\alpha$ in $T_\alpha$, namely a set of coset
representatives of $T'_\alpha$ in $T_\alpha$, and define $\sigma$ on
$T''_\alpha$ by choosing arbitrarily for each $t''\in T''_\alpha$ a
$\pi$-preimage in $B$. Extend then $\sigma$ to $T_\alpha$ by
$\sigma(t''t')=\sigma(t'')\sigma(t')$ for $t''\in T''_\alpha,t'\in
T'_\alpha$.

Let $F$ be a locally finite group of cardinality $>\#X$, and fix an
imbedding of $X$ in $F\setminus\{1\}$.  As a first step, we consider
the group $G_0=B\wrf(T\times F)$, and define a map $\Phi_0\colon B\to
G_0$ as follows:
\begin{equation}\label{eq:Phi_0}
  \Phi_0(b)=(\phi,\pi(b),1)\text{ with }\phi(t,f)=\begin{cases}
    b & \text{ if }f=1,\\
    \theta_f(\sigma(t)b\sigma(t\pi(b))^{-1}) & \text{ if }f\in X,\\
    1 & \text{ otherwise.}
  \end{cases}
\end{equation}

\begin{lemma}\label{lem:Phi_0}
  The map $\Phi_0$ is well-defined and is an injective homomorphism into
  $G_0$.
\end{lemma}
\begin{proof}
  To see that $\Phi_0$ is well-defined, note that the argument
  $\sigma(t)b\sigma(t\pi(b))^{-1}$ belongs to $\ker(\pi)=C$, so that $\theta_f$
  may be applied to it.

  We next show that the image of $\Phi_0$ belongs to $G_0$. Consider
  $b\in B$. Let $\alpha$ be such that $\pi(b)$ belongs to the finite
  group $T_\alpha$. Now given $t\in T$, let $\omega\in I$ be such that
  $t\in T_\omega$. Write $t$ using $T_\alpha$ and transversal elements
  as $t''_\omega\dots t''_\beta u$ with $\omega>\dots>\beta>\alpha$
  and $t''_\omega\in T''_\omega,\dots,t''_\beta\in T''_\beta,u\in
  T_\alpha$. Then
  $\sigma(t)=\sigma(t''_\omega)\dots\sigma(t''_\beta)\sigma(u)$ and
  $\sigma(t\pi(b))=\sigma(t''_\omega)\dots\sigma(t''_\beta)\sigma(u\pi(b))$,
  so that $\sigma(t)b\sigma(t\pi(b))^{-1}$ is conjugate to
  $\sigma(u)b\sigma(u\pi(b))^{-1}$, and therefore
  $\theta_f(\sigma(t)b\sigma(t\pi(b))^{-1})=\theta_f(\sigma(u)b\sigma(u\pi(b))^{-1})$
  takes only finitely many values because $\theta_f$ vanishes on
  $[B,B]$. Also, $\theta_f(\sigma(t)b\sigma(t\pi(b))^{-1})=1$ except
  for finitely many values of $f\in X$. In summary, the function
  $\phi\in B^{T\times F}$ is such that $\phi(t,f)$ takes only finitely
  many values.

  It is clear that $\Phi_0$ is injective: if $b\neq1$ and
  $\Phi_0(b)=(\phi,\pi(b),1)$ then $\phi(1,1)=b\neq1$. It is a
  homomorphism because all $\theta_f$ are homomorphisms.
\end{proof}

\begin{lemma}\label{lem:C imbeds}
  We have $\Phi_0(C)\le[G_0,G_0]$.
\end{lemma}
\begin{proof}
  If $b\in[B,B]$ then clearly $\Phi_0(b)\in[G_0,G_0]$. Since $C$ is
  generated by $[B,B]\cup\{b_x\}_{x\in X}$, it suffices to consider
  $b=b_x$.

  We define $g\in G_0$ by
  \[g=(\psi,1,1)\text{ with }\psi(t,f)=\begin{cases}b_x & \text{ if }f=1,\\ 1 & \text{ otherwise}.\end{cases}
  \]
  Then $\Phi_0(b_x)=(\phi,1,1)$ with $\phi(t,1)=b_x$ and
  $\phi(t,x)=b_x^{-1}$, all other values being trivial, according
  to~\eqref{eq:Phi_0}; so, as was to be shown,
  \[\Phi_0(b_x)=(\phi,1,1)=({}^x\psi^{-1}\cdot\psi,1,1)=[(1,1,x^{-1}),g]\in[G_0,G_0].\qedhere\]
\end{proof}

\noindent We next define
\[G=G_0\wrf(\Q/\Z)\]
and a map $\Phi\colon B\to G$ by
\[\Phi(b)=(\phi,0)\text{ with }\phi(r)=\Phi_0(b)\text{ for all }r\in\Q/\Z.\]
\begin{lemma}\label{lem:B imbeds}
  The map $\Phi$ is an injective homomorphism, and
  $\Phi(B)\le[G,G]$.
\end{lemma}
\begin{proof}
  Clearly $\Phi$ is an injective homomorphism, since $\Phi_0$ is an
  injective homomorphism by Lemma~\ref{lem:Phi_0}.

  We identify $\Q/\Z$ with $\Q\cap[0,1)$. For every $n\in\N$, consider
  the map $\Psi_n\colon B\to G$ defined by
  \[\Psi_n(b)=(\phi,0)\text{ with }\phi(r)=\begin{cases}
    \Phi_0(b) & \text{ if }r\in[0,1/n),\\
    1 & \text{ otherwise;}
  \end{cases}
  \]
  so $\Phi=\Psi_1$.  We know from Lemma~\ref{lem:C imbeds} that
  $\Psi_n(C)$ is contained in $[G,G]$.

  Consider now $b\in B$. Since $B/C$ is torsion, there exists $n\in\N$
  such that $b^n\in C$. We define $g\in G$ by
  \[g=(\psi,0)\text{ with }\psi(r)=\Phi_0(b)^{\lfloor rn\rfloor}\text{
    for }r\in[0,1)\cap\Q.
  \]
  Let us write $h=\Phi_0(b)$, and consider the element
  $[(1,1/n),g]\cdot\Psi_n(b^n)=(\phi,0)$.  If $r\in[0,1/n)$ then
  $\phi(r)=\psi(r-1/n)^{-1}\psi(r)h^n=h$, while if $r\in[1/n,1)$ then
  $\phi(r)=\psi(r-1/n)^{-1}\psi(r)=h$; therefore
  \[\Phi(b)=[(1,1/n),g]\cdot\Psi_n(b^n)\in[G,G].\qedhere\]
\end{proof}

\begin{proof}[Proof of Proposition~\ref{prop:imbed'}]
  The first assertion is simply Lemma~\ref{lem:B imbeds}.

  Assume that $B$ has locally subexponential growth, and consider a
  finite subset $S$ of $G$. Then there exists a subgroup of $G$ that
  contains $S$ and is virtually a finite power of $B$, hence has
  subexponential growth. This shows that $G$ has locally
  subexponential growth.

  For the last assertion: if $B$ is infinite, we wish to find a
  subgroup $H$ of $G$ with the same cardinality as $B$, such that
  $\Phi$ maps into $[H,H]$. For each $b\in B$, choose a finite subset
  $S_b$ of $G$ such that $\Phi(b)\in[\langle S_b\rangle,\langle
  S_b\rangle]$, and a subgroup $G_b$, containing $S_b$, that is
  virtually a finite power of $B$. Consider the group $H$ generated by
  the union of all the $G_b$. As soon as $B$ is infinite, all $G_b$
  have the same cardinality as $B$, and so does $H$.
\end{proof}

\section{Orbits and inverted orbits}

Let $G=\langle S\rangle$ be a finitely generated group acting on the
right on a set $X$. We consider $X$ as a the vertex set of a graph
still denoted $X$, with for all $x\in X,s\in S$ an edge labelled $s$
from $x$ to $xs$. We denote by $d$ the path metric on this graph.

\begin{definition}
  A sequence $(x_0,x_1,\dots)$ in $X$ is \emph{spreading} if for all
  $R$ there exists $N$ such that if $i,j\ge N$ and $i\neq j$ then
  $d(x_i,x_j)\ge R$.
\end{definition}

\begin{example}
  If all $x_i$ lie in order on a geodesic ray starting from $x_0$ (for
  example if $X$ itself is a ray starting from $x_0$) and for all $i$
  we have $d(x_0,x_{i+1})\ge 2 d(x_0,x_i)$, then $(x_i)$ is spreading.
\end{example}

\begin{lemma}
  Equivalently, a sequence $(x_0,x_1,\dots)$ in $X$ is spreading if
  and only if for all $R$ there exists $N$ such that if $i\neq j$ and
  $i\ge N$ then $d(x_i,x_j)\ge R$.
\end{lemma}
\begin{proof}
  Assume the converse, namely $d(x_i,x_j)<R$ along a sequence with
  $i\to\infty$ and $j\not\to\infty$. Then, up to passing to a
  subsequence, $j$ may be assumed constant. There are then
  $i,i'\to\infty$ with $i\neq i'$ and $d(x_i,x_{i'})<2R$, so $(x_i)$
  is not spreading.
\end{proof}

\begin{definition}
  A sequence $(x_i)$ in $X$ \emph{locally stabilises} if for all $R$
  there exists $N$ such that if $i,j\ge N$ then the $S$-labelled
  radius-$R$ balls centered at $x_i$ and $x_j$ in $X$ are equal.
\end{definition}

\begin{definition}\label{def:rectifiable}
  A sequence of points $(x_i)$ in $X$ is \emph{rectifiable} if for all
  $i,j$ there exists $g \in G$ with $x_i g =x_j$ and $x_k g\ne x_\ell$
  for all $k\notin\{i,\ell\}$.
\end{definition}

For example, if $X=\Z$ and $G=\Z$ acting by translations, then
$\Sigma=\{2^i:i\in\N\}$ is rectifiable, since $2^j-2^i=2^\ell-2^k$
only has trivial solutions $i=k,j=\ell$ and $i=j,k=\ell$.

\begin{remark}
  The sequence $\Sigma=(x_i)\subseteq X$ is rectifiable if and only if
  for all $i,j$ there exists $g\in G$ with $x_i g = x_j$ and
  $\Sigma\cap\Sigma g\subseteq\{x_j\}\cup\operatorname{fixed.\!points}(g)$.
\end{remark}

The following property is closer to Olshansky-Osin's notion of
parallelogram-free
sequence~\cite{olshanskii-osin:qiembedding}*{Definition~2.3}:
\begin{definition}
  Fix a point $z\in X$. A sequence $(g_i)$ in $G$ is
  \emph{parallelogram-free at $z$} if, for all $i,j,k,\ell$ with $i\neq
  j$ and $j\neq k$ and $k\neq\ell$ and $\ell\neq i$ one has
  $z g_i^{-1}g_j g_k^{-1} g_\ell\ne z$.
\end{definition}

\begin{lemma}\label{lem:parallelogram}
  If $z\in X$ and $(g_i)$ is parallelogram-free at $z$, then $(z
  g_i^{-1})$ is a rectifiable sequence in $X$.
\end{lemma}
\begin{proof}
  Set $x_i=z g_i^{-1}$ for all $i\in\N$. Given $i,j\in\N$, consider
  $g=g_i g_j^{-1}$, so $x_i g = x_j$. If furthermore we have $x_k
  g=x_\ell$, then we have $z g_k^{-1} g_i g_j^{-1} g_\ell = z$, so
  either $k=i$, or $i=j$ which implies $k=\ell$, or $j=\ell$ which
  implies $k=i$, or $\ell=k$. In all cases $k\in\{i,\ell\}$ as was to
  be shown.
\end{proof}

It is clear that, if $G$ is finitely generated and $X$ is infinite,
then it admits spreading and locally stabilizing sequences. Also, a
subsequence of a spreading or locally stabilizing sequence is again
spreading, respectively locally stabilizing. We give in the next
section a general construction of rectifiable sequences, and
in~\S\ref{ss:grigorchuk} a concrete example in the first Grigorchuk
group.

\subsection{Separating actions}
Consider a group $G$ acting on a set $X$. We recall that the
\emph{fixator} of the subset $Y\subseteq X$ is the set
$\operatorname{Fix}(Y):=\{g\in G\colon yg=y\text{ for all }y\in Y\}$.

\begin{definition}[Ab\'ert~\cite{abert:nonfree}]
  The group $G$ \emph{separates} $X$ if for every finite subset
  $Y\subseteq X$ and every $y_0\notin Y$ there exists
  $g\in\operatorname{Fix}(Y)$ with $y_0 g\ne y_0$.
\end{definition}

\begin{lemma}\label{lem:separating=>rectifiable}
  Let $G$ be a group acting on a non-empty set $X$ and separating
  it. Then there exists a rectifiable sequence $(x_i)$ in $X$.
\end{lemma}
\begin{proof}
  We choose an arbitrary point $z\in X$, and construct iteratively a
  parallelogram-free sequence $(g_i)$ at $z$; by
  Lemma~\ref{lem:parallelogram}, this proves the lemma.  Suppose that
  we have already constructed $g_j$ for all $j<i$. For $i\ge 0$, we
  then construct $g_i$ in the following way. We define
  \[X_i^r:=\{z g_{j_1}\cdots g_{j_s} : s\le r,\,0\le j_1,\dots,j_s<i\}\setminus\{z\}\qquad\text{for }r=1,2,3,
  \]
  and choose an element $g_i\in\operatorname{Fix}(X_i^3)$ that moves
  $z$. In particular, $z g_i^{-1}\not\in X_i^3\cup\{z\}$.

  Let us suppose that we have $z g_k^{-1} g_i g_j^{-1} g_\ell = z$
  with $i\neq j\neq k\neq\ell\neq i$, and seek a contradiction. If
  needed, we switch $i\leftrightarrow j$ and $k\leftrightarrow\ell$
  and consider the equivalent equality $z g_\ell^{-1} g_j g_i^{-1} g_k
  = z$, to reduce to the case $k<\ell$.

  We note, first, $u := z g_k^{-1}\in X_\ell^1$, because
  $k<\ell$. Next, we consider $v := z g_k^{-1} g_i = u g_i$ and claim
  $v\in X_\ell^2$. By assumption $i\neq k$ so $v\neq z$; if $i<\ell$
  then $v\in X_\ell^2$ by definition of $X_\ell^2$, while if
  $i\ge\ell$ then $i>k$ so $u\in X_i^1$ and $v=u g_i=u\in
  X_\ell^1\subseteq X_\ell^2$. Finally, we consider $w = z g_k^{-1}
  g_i g_j^{-1} = v g_j^{-1}$. By assumption $j\ne\ell$; if $j<\ell$
  then $w\in X_\ell^3\cup\{z\}$ by definition of $X_\ell^3$, while if
  $j>\ell$ then $v\in X_j^2$ so $w=v g_j^{-1}=v\in X_\ell^2\subseteq
  X_\ell^3$. We have in both cases reached the contradiction
  $w = z g_\ell^{-1}\in X_\ell^3\cup\{z\}$.
\end{proof}

We quote from Ab\'ert~(\cite{abert:nonfree}*{Proof of Corollary~1.4},
see also~\cite{bartholdi-erschler:orderongroups}*{Lemma~6.11}) that
the action of weakly branch groups separates the boundary of their
tree.  Since the first Grigorchuk group is weakly branched
(see~\cite{grigorchuk:onbranch}*{Theorem~1}
or~\cite{bartholdi-g-s:bg}*{Proposition~1.25}), it provides by
Lemma~\ref{lem:separating=>rectifiable} an example of a group action
with rectifiable sequences. We also see it directly in the following
section.

\subsection{An orbit for the first Grigorchuk group}\label{ss:grigorchuk}
In this subsection, we consider the first Grigorchuk group
$\grig=\langle a,b,c,d\rangle$. Recall that it acts on set of infinite
sequences $\{\mathbf0,\mathbf1\}^\infty$ over a two-letter alphabet,
which is naturally the boundary of a binary rooted tree; the action
may be found in~\cite{grigorchuk:growth} and
in~\cite{bartholdi-g-s:bg}*{\S1.6.1}. We denote by
$X=\mathbf1^\infty\grig$ the orbit of the rightmost ray, and view it
as a graph with vertex set $X$ and for each $x\in X$ and each
generator $s$ of $\grig$ an edge labeled $s$ from $x$ to $xs$; such
graphs are called \emph{Schreier} graphs. We construct explicitly a
spreading, locally stabilizing, rectifiable sequence for the action of
$\grig$ on $X$: for all $i\in\N$, let us define
\[x_i=\mathbf0^i\mathbf1^\infty.
\]
The geometric image of the Schreier graph $X$ is that of a
half-infinite line. The point $x_i$ is at position $2^i$ along this
ray.

\begin{lemma}
  For all $i,j\in\N$,
  \begin{enumerate}
  \item the marked balls of radius $2^{\min(i,j)}$ in $X$ around $x_j$
    and $x_i$ coincide;
  \item the distance $d(x_i,x_j)$ is $|2^i-2^j|$;
  \item there exists $g_{i,j}\in\grig$ of length $|2^i-2^j|$ with $x_i
    g_{i,j} = x_j$ and $x_k g_{i,j}\neq x_\ell$ for all $(k,\ell)\neq
    (i,j)$.
  \end{enumerate}
\end{lemma}

\begin{proof}
  (1,2) Consider the map $\sigma: a\mapsto c,b\mapsto d^a,c\mapsto
  b^a,d\mapsto c^a$. It defines a self-map of $X$ by sending
  $\mathbf1^\infty g$ to $\mathbf1^\infty\sigma(g)$. A direct
  calculation shows that it sends $x\in X$ to $\mathbf0x$.

  Since $\sigma$ is $2$-Lipschitz on words of even length in
  $\{a,b,c,d\}$, it maps the ball of radius $n$ around $x$ to the ball
  of radius $2n$ around $\mathbf0x$. Its image is in fact a net in the
  ball of radius $2n$: two points at distance $1$ in the ball of
  radius $n$ around $x$ will be mapped to points at distance $1$ or
  $3$ in the image, connected either by a path $a$ or by a segment
  $a-b-a$, $a-c-a$ or $a-d-a$. In particular, the $2^n$-neighbourhoods
  of the balls about the $x_m$ coincide for all $m\ge n$.

  (3) Note, first, that there exists $g_{i,j}$ with $x_i g_{i,j} =
  x_j$, because the rays ending in $\mathbf1^\infty$ form a single
  orbit. Note, also, that we have $x_k g_{i,j} = x_\ell$ for either
  finitely many $(k,\ell)\neq(i,j)$ or for all but finitely many
  $(k,\ell)$, because there is a level $N$ at which the decomposition
  of $g_{i,j}$ consists entirely of generators; if the entry at
  $\mathbf0^N$ of $g_{i,j}$ is trivial or `$d$' then all but finitely
  many of the $x_k$ are fixed; while otherwise (up to increasing $N$
  by at most one) we may assume it is an `$a$'; then $\mathbf0^{N+1}
  g_{i,j}=\mathbf0^N\mathbf1$, so $x_k\neq x_\ell$ for all $k>N+1$.

  We use the following property of the Grigorchuk group: for every
  finite sequence $u\in\{\mathbf0,\mathbf1\}^*$ there exists an
  element $h_u\in\grig$ whose fixed points are precisely those
  sequences in $\{\mathbf0,\mathbf1\}^\infty$ that do not start with
  $u$; see~\cite{bartholdi-g-s:bg}*{Proposition~1.25}.

  If the entry at $\mathbf0^N$ of $g_{i,j}$ is trivial, then we
  multiply $g_{i,j}$ with $h_{\mathbf0^M}$ for some $M>\max(N,i)$, so
  as to fall back to the second case.

  Then, for each pair $(k,\ell)\neq(i,j)$ with $x_k g_{i,j}=x_\ell$,
  we multiply $g_{i,j}$ with $h_{\mathbf0^\ell\mathbf1}$, so as to
  destroy the relation $x_k g_{i,j}=x_\ell$.

  The resulting element $g_{i,j}$ satisfies the required conditions.
\end{proof}

\section{Subexponential growth of wreath products}

In this section, we show how some permutational wreath products have
subexponential growth.

\begin{definition}
  The group $G$ acting on $X$ has the \emph{subexponential wreathing
    property} if for any finitely generated group of subexponential
  growth $H$ the restricted wreath product $H\wr_X G$ has
  subexponential growth.
\end{definition}

\begin{lemma}\label{lem:concavehull}
  Let $f$ be a positive sublinear function, namely $f(n)/n\to 0$ as
  $n\to\infty$. Then $f$ is bounded from above by a concave sublinear
  function.
\end{lemma}
\begin{proof}
  For every $\theta\in(0,1)$, let $n_\theta$ be such that $f(n)-\theta
  n$ is maximal. Given $n\in\R$, let $\zeta<\theta$ be such that
  $n\in[n_\theta,n_\zeta]$ with maximal $\zeta$ and minimal $\theta$,
  and define $\overline f(n)$ on $[n_\theta,n_\zeta]$ by linear
  interpolation between $(n_\theta,f(n_\theta))$ and
  $(n_\zeta,f(n_\zeta))$. Clearly $\overline f\ge f$, and $\overline
  f(n)/n$ is decreasing and coincides infinitely often with $f(n)/n$,
  so it converges to $0$.
\end{proof}

\begin{lemma}\label{lem:swp}
  Let the Schreier graph of $X$ have linear growth, and assume that
  $G$ has sublinear inverted orbit growth on $X$.  Assume also that
  $G$ has subexponential growth.  Then $G$ has the subexponential
  wreathing property.
\end{lemma}
\begin{proof}
  We essentially
  follow~\cite{bartholdi-erschler:permutational}*{Lemma~5.1}.

  Fix some $x_0\in X$ and let $\rho(n)$ be the growth of inverted
  orbits starting from $x_0$. By assumption, $\rho(n)/n\to0$, and
  there is a constant $C$ such that the ball of radius $n$ around
  $x_0$ has cardinality $\le Cn$.

  Let $H$ be a group of subexponential growth, and choose a finite
  generating set for $H$. By Lemma~\ref{lem:concavehull}, there exists
  a log-concave subexponential function $\overline v_H$ bounding the
  growth function $v_H(n)$ of $H$.

  We view $H\wr_X G$ as generated by the generating set of $G$ and the
  imbedding of the generating set of $H$ as functions supported at
  $\{x_0\}$.

  Consider an element $(c,g)\in H\wr_X G$ of norm $R$. The function
  $c\colon X\to H$ has support of cardinality $k\le\rho(R)$, and this
  support is contained in the ball of radius $R$ around $x_0$. Since
  the ball of radius $R$ has cardinality at most $CR$, the number of
  possible choices for this support is at most $\binom{CR}{\rho(R)}$.
  Let $\{z_1,\dots,z_k\}$ denote the support of $c$. The values of $c$
  belong to $H$, and their total norm is $\le R$, so the number of
  choices for $c$ is at most $v_H(n_1)\cdots v_H(n_k)$ subject to the
  constraint $n_1+\cdots+n_k\le R$. Since $v_H(n_i)\le\overline
  v_H(n_i)$ and $\overline v_H(n_i)$ is log-concave, the number of
  choices for $v_H$ is at most $\overline v_H(R/k)^k$. On the other
  hand, the number of choices for $g$ is at most $v_G(R)$. All in all,
  the cardinality of the ball of radius $R$ in $H\wr_X G$ is bounded
  from above as
  \[v_{H\wr_X G}(R)\le v_G(R) \binom{CR}{\rho(R)} \overline
  v_H\Big(\frac{R}{\rho(R)}\Big)^{\rho(R)}.
  \]
  Since it is a product of subexponential functions, it is itself
  subexponential.
\end{proof}

We now
quote~\cite{bartholdi-erschler:permutational}*{Proposition~4.4}: the
inverted orbit growth of the first Grigorchuk group $\grig$ on
$X=\mathbf1^\infty\grig$ is sublinear (actually of the form $n^\alpha$
for some $\alpha<1$); therefore, by Lemma~\ref{lem:swp}, the action of
$\grig$ on $X$ has the subexponential wreathing property.  (It follows
from~\cite{bartholdi-erschler:givengrowth} that all Grigorchuk groups
$G_\omega$ also have the subexponential wreathing property, as soon as
$\omega\in\{0,1,2\}^\infty$ contains infinitely many copies of each
symbol.)

\section{The construction of $W$}\label{sec:W}

Using the results of the previous section, we select a group $G$
acting on a set $X$, and a rectifiable, spreading, locally stabilizing
sequence $(x_i)$ of elements of $X$.

Let $(b_1,b_2,\dots)$ be a sequence in $B$. We will specify later a
rapidly increasing sequence $0\le n(1) < n(2) <\dots$; assuming this
sequence given, we define $f\colon X\to B$ by
\[f(x_{n(1)})=b_1,\qquad f(x_{n(2)})=b_2,\qquad \dots,\qquad f(x)=1\text{ for other }x.
\]
We then consider the subgroup $W=\langle G,f\rangle$ of the
unrestricted wreath product $B^X\rtimes G$.

\begin{lemma}\label{lem:contains [B,B]}
  Denote by $B_0$ the subgroup of $B$ generated by
  $\{b_1,b_2,\dots\}$. If the sequence $(x_i)$ is rectifiable, then $[W,W]$
  contains $[B_0,B_0]$ as a subgroup.
\end{lemma}
\begin{proof}
  Without loss of generality and to lighten notation, we rename $B_0$
  into $B$. We also denote by $\iota\colon B\to B^X\rtimes G$ the
  imbedding of $B$ mapping the element $b\in B$ to the function $X\to
  B$ with value $b$ at $x_0$ and $1$ elsewhere.  We shall show that
  $[W,W]$ contains $\iota([B,B])$. For this, denote by $H$ the
  subgroup $\iota(B)\cap W$.

  We first consider an elementary commutator $g=[b_i,b_j]$. Let
  $g_i,g_j\in G$ respectively map $x_i,x_j$ to $x_0$, and be such that
  $g_i g_j^{-1}$ maps no $x_k$ to $x_\ell$ with $k\neq\ell$, except
  for $x_i g_i g_j^{-1}=x_j$. Consider $[f^{g_i},f^{g_j}]\in [W,W]$;
  it belongs to $B^X$, and has value $[b_i,b_j]$ at $x_0$ and is
  trivial elsewhere, so equals $\iota(g)$ and therefore $\iota(g)\in
  H$.

  We next show that $H$ is normal in $B^X$. For this, consider $h\in
  H$. It suffices to show that $h^{\iota(b_i)}$ belongs to $H$ for all
  $i$. Now $h^{\iota(b_i)}=h^{f^{g_i}}$ belongs to $H$, and we are
  done.
\end{proof}

\begin{proposition}\label{prop:Wconvergence}
  Let $G$ be a group acting on $X$. Let the sequence $(x_i)$ in $X$ be
  spreading and locally stabilizing. Let a sequence of elements
  $(b_i)$ be given in the group $B$, all of the same order
  $\in\N\cup\{\infty\}$.

  Then for every increasing sequence $(m(i))$ there is a choice of
  increasing sequence $(n(i))$ with the following property.

  For all $i\in\N$, let $f_i$ be the finitely supported function $X\to
  B$ with $f_i(x_{n(j)})=b_j$ for all $j\le i$, all other values being
  trivial, and denote by $W_i$ the group $\langle f_i,G\rangle$.  Let
  also $f\colon X\to B$ be defined by $f(x_{n(j)})=b_j$ for all
  $j\in\N$, all other values being trivial, and write $W=\langle
  G,f\rangle$.

  Then the ball of radius $m(i)$ in $W$ coincides with the ball of
  radius $m(i)$ in $W_i$, via the identification $f\leftrightarrow
  f_i$.

  Furthermore, the term $n(i)$ depends only on the previous terms
  $n(1),\dots,n(i-1)$, on the initial terms $m(1),\dots,m(i-1)$, and
  on the ball of radius $m(i)$ in the subgroup $\langle
  b_1,\dots,b_{i-1}\rangle$ of $B$.
\end{proposition}
\begin{proof}
  Choose $n(i)$ such that $d(x_j,x_k)\ge m(i)$ for all $j\neq k$ with
  $k\ge n(i)$, and such that the balls of radius $m(i)$ around
  $x_{n(i)}$ and $x_j$ coincide for all $j>n(i)$.

  Consider then an element $h\in W$ in the ball of radius $m(i)$, and
  write it in the form $h=(c,g)$ with $c\colon X\to B$ and $g\in
  G$. The function $c$ is a product of conjugates of $f$ by words of
  length $<m(i)$. Its support is therefore contained in the union of
  balls of radius $m(i)-1$ around the $x_j$, with $j$ either $\ge
  n(i)$ or of the form $n(k)$ for $k<i$. In particular, the entries of
  $c$ are in $\langle b_1,\dots,b_{i-1}\rangle\cup\bigcup_{j\ge
    i}\langle b_j\rangle$. For $j>n(i)$, the restriction of $c$ to the
  ball around $x_j$ is determined by the restriction of $c$ to the
  ball around $x_{n(i)}$, via the identification $b_i\mapsto b_j$,
  because the neighbourhoods in $X$ coincide and all cyclic groups
  $\langle b_j\rangle$ are isomorphic.

  It follows that the element $h\in W$ is uniquely determined by the
  corresponding element in $W_i$.
\end{proof}

  
\begin{corollary}\label{cor:subexp}
  Let $G$ be a group acting on $X$ with the subexponential wreathing
  property. Let the sequence $(x_i)$ be spreading and locally
  stabilizing.

  If $B$ has locally subexponential growth, then there exists a
  sequence $(n(i))$ such that the group $W$ has subexponential growth.
%
\end{corollary}
\begin{proof}
  Let $Z=\langle z\rangle$ be a cyclic group whose order (possibly
  $\infty$) is divisible by the order of the $b_i$'s.  We replace $B$
  by $B\times Z$ and each $b_i$ by $b_i z$, so as to guarantee that
  all generators in $B$ have the same order.

  Let $\epsilon_i$ be a decreasing sequence tending to $1$. We
  construct a sequence $m(i)$ inductively, and obtain the sequence
  $n(i)$ by Proposition~\ref{prop:Wconvergence}, making always sure
  that $m(i)$ depends only on $m(j),n(j)$ for $j<i$.

  Denote by $v_i$ the growth function of the group $W_i$ introduced in
  Proposition~\ref{prop:Wconvergence}. Since the group $W_i$ is
  contained in $B\wr_X G$ and $G$ has the subexponential wreathing
  property, it has subexponential growth. Therefore, there exists
  $m(i)$ be such that
  \[v_i(m(i))\le \epsilon_i^{m(i)}.
  \]
  By Proposition~\ref{prop:Wconvergence}, the terms
  $n(i+1),n(i+2),\dots$ can be chosen in such a manner that the balls
  of radius $m(i)$ coincide in $W$ and $W_i$.

  Denote now by $w$ the growth function of $W$. We then have
  $w(m(i))\le\epsilon_i^{m(i)}$. Therefore,
  \[w(R)\le\epsilon_i^{R+m(i)}\text{ for all }R>m(i),\] so
  $\lim\sqrt[R]{w(R)}\le\epsilon_i$. Since this holds for all $i$, the
  growth of $W$ is subexponential.
%
%
%
%
\end{proof}

\begin{proof}[Proof of Theorem~\ref{thm:imbed}]
  By Proposition~\ref{prop:imbed'}, the countable, locally
  subexponentially growing group $B$ imbeds in $[H,H]$ for a
  countably, locally subexponentially growing group $H$. By
  Lemma~\ref{lem:contains [B,B]}, $[H,H]$ imbeds in $[W,W]$, and by
  Corollary~\ref{cor:subexp}, the finitely generated group $W$ has
  subexponential growth.
\end{proof}

\begin{remark}
  If the sequence $(x_i)$ is only spreading, or only stabilizing, then
  it may happen that $W$ have exponential growth, even if the sequence
  $(n(i))$ grows arbitrarily fast.
\end{remark}
\begin{proof}
  We first consider an example where the sequence $(x_i)$ is spreading
  but not stabilizing. Consider $G=\grig$ acting on $X=\mathbf1^\infty
  G$, and let $P$ denote the stabilizer of $\mathbf1^\infty$ so that
  $X=P\backslash G$. Since the action is faithful, we have
  $\bigcap_{g\in G}P^g=1$, and in fact $\bigcap_{g\in T}P^g=1$ for a
  sequence $T$ in $G$ such that $(\mathbf1^\infty t\colon t\in T)$ is
  spreading. Take $B=\langle z\rangle\cong\Z$ and define $f\colon X\to
  B$ by $f(\mathbf1^\infty t)=z$ for all $t\in T$, all other values
  being $1$. Then $\langle G,f\rangle\cong\Z\wr G$ has exponential
  growth.

  We next consider an example where the sequence $(x_i)$ is
  stabilizing but not spreading. Again, consider $G=\grig$ acting on
  $X$, and consider a spreading, stabilizing sequence $(x_{2i})$ in
  $X$. Set $x_{2i-1}=x_i a$. Consider $B=\grig$, and note that, since
  $G$ does not satisfy any law, there are sequences
  $(g_1,h_1),(g_2,h_2),\dots$ of pairs of elements of $G$ such that
  the groups $\langle g_i,h_i\rangle$ converge to a free group of rank
  $2$ in the Cayley topology. Set then $f(x_{2i})=h_i$ and
  $f(x_{2i-1})=g_i$, and note that $\langle G,f\rangle$ contains a
  free group.
\end{proof}

\begin{remark}\label{rem:2gen}
  If $B$ has locally subexponential growth, then it may be imbedded in
  a $2$-generator group of subexponential growth.
\end{remark}
\begin{proof}
  We make the following general claim about finitely generated groups:
  if $W$ is finitely generated, then there exists a $2$-generated
  hyper-$W$ group in which $[W,W]$ imbeds. Since the imbedding given
  by Theorem~\ref{thm:imbed} is actually into $[W,W]$, this is
  sufficient to prove the remark.

  Let us now turn to the claim, and consider a group $W$ generated by
  a set $S=\{s_1,\dots,s_n\}$. Let $C=\langle t|t^{2^n}\rangle$ be a
  cyclic group, and consider the subgroup
  \[\overline W=\langle x,t\rangle\le W\wr C,\]
  with $x\colon C\to W$ defined by $x(t^{2^{i-1}})=s_i$ for all
  $i\in\{1,\dots,n\}$, all other values being trivial. The imbedding
  of $[W,W]$ into $\overline W$ is as functions $x\colon C\to W$ whose
  support is contained in $\{t\}$. Indeed, given $w\in[W,W]$, write it
  as a balanced word (i.e.\ with exponent sum zero in each variable)
  $w$ over $S$, and replace each $s_i$ by
  $\overline{s_i}:=x^{t^{1-2^{i-1}}}$, yielding $\overline
  w\in[\overline W,\overline W]$. The functions $\overline{s_i}\colon
  C\to W$ all have disjoint supports, except at $t$ where their
  respective value is $s_i$. Therefore, $\overline w$ is supported
  only at $t$ and has value $w$ there.
\end{proof}
%
%
%


\end{document}